\documentclass[12pt]{article}

\usepackage[top=1in,bottom=1in,left=1in,right=1in]{geometry}

\usepackage{indentfirst}
\usepackage{latexsym}
\usepackage{algorithm}
\usepackage{algorithmic}
\usepackage{setspace}
\usepackage{longtable, booktabs}
\usepackage{tabu}
\usepackage{caption}
\usepackage{amsthm}
\usepackage{amsmath}
\usepackage{amssymb}
\usepackage{supertabular}

\usepackage{mathrsfs}
\usepackage{url}

\usepackage{rotating}
\usepackage{multicol}
\usepackage{multirow}
\usepackage{makeidx} 
\usepackage{epsfig}
\usepackage{fancyhdr}       

\allowdisplaybreaks[4]

\newtheorem{theorem}{Theorem}[section]
\newtheorem{lemma}{Lemma}[section]

\newtheorem{corollary}[lemma]{Corollary}
\newtheorem{remark}[lemma]{Remark}

\usepackage{hyperref}
\hypersetup{
	hypertex=true,
	colorlinks=true,
	linkcolor=blue,
	anchorcolor=blue,
	citecolor=blue
}
\def\lam{\lambda}

\def\Soc{{\rm Soc}}

\def\D{\mathcal{D}}
\def\B{\mathcal{B}}
\def\P{\mathcal{P}}

\def\B{\mathcal{B}}

\DeclareMathOperator{\Aut}{Aut} 
 \DeclareMathOperator{\Out}{Out}


\begin{document}
	
	\title{Block-transitive $t$-$(k^2,k,\lam)$ designs and simple exceptional groups of Lie type}
	\author{Xingyu Chen$^1$,Haiyan Guan$^{1,2}$\footnote{This work is
			supported by the National Natural Science Foundation
			of China (Grant No.12271173 ).} \\
		{\small\it  1. College of Mathematics and Physics, China Three Gorges University, }\\
		{\small\it  Yichang, Hubei, 443002, P. R. China}\\
		{\small\it 2. Three Gorges Mathematical Research Center, China Three Gorges University,}\\
		{\small\it  Yichang, Hubei, 443002, P. R. China}\\
		\date{}
	}
	
	\maketitle

	\date

	\begin{abstract}Let $G$ be an automorphism group of a nontrivial $t$-$(k^2,k,\lambda)$ design. In this paper, we prove that if $G$ is block-transitive, then the socle of $G$ cannot be a finite simple exceptional group of Lie type.
		
		
		\medskip \noindent{\bf MR(2010) Subject Classification} 05B05, 05B25, 20B25
	\end{abstract}

\section{Introduction}

A $t$-$(v,k,\lambda)$ design $\mathcal{D}$ is a pair $(P,\mathcal{B})$, where $P$ is a set of $v$ points and $\mathcal{B}$ is a family of $b$ $k$-subsets (called blocks) of $P$, such that any $t$-subset of $P$ is contained in exactly $\lambda$ blocks. It is \textit{nontrivial} if $t<k<v$. An \textit{automorphism} of a design $\mathcal{D}$ is a permutation of the point set that preserves the block set. The group of all automorphisms of $\mathcal{D}$ is denoted by $\mathrm{Aut}(\mathcal{D})$, and any subgroup of $\mathrm{Aut}(\mathcal{D})$ is called an \textit{automorphism group} of $\D$. Let $G\leq \mathrm{Aut}(\mathcal{D})$, if $G$ acts transitively on the set of points (resp.~blocks), then the design $\mathcal{D}$ is said to be \textit{point-transitive} (resp.~\textit{block-transitive}), and $\mathcal{D}$ is called \textit{point-primitive} if $G$ is a primitive permutation group on the point set $P$. The socle of a finite group $G$, denoted by $\mathrm{Soc}(G)$, is the product of all its minimal normal subgroups. A finite group $G$ is called \textit{almost simple} if its socle $X$ is a non-abelian simple group, and $X\trianglelefteq G \leq \Aut(X)$.

In \cite{MR4481043}, Montinaro and Francot demonstrated that for a $2$-$(k^2,k,\lam)$ design with $\lam\mid k$, the flag-transitive automorphism group $G$ is either an affine group or an almost simple group, and possible designs have been studied in \cite{MR4561672,MR4516389,MR4481043}. In \cite{MR4830077}, Guan and Zhou generalized this result, showing that an automorphism group $G$ of a block-transitive $t$-$(k^2,k,\lam)$ design must be point-primitive, and   $G$ is either an affine group or an almost simple group. For a primitive permutation group of almost simple type, its socle can be an alternating group, a classical simple group, a simple exceptional group of Lie type or a sporadic simple group. The cases where $X$ is sporadic or alternating have been studied in \cite{MR4830077}. In this paper, we continue to study the almost simple block-transitive $t$-$(k^2,k,\lam)$ designs with the socle $X$ of $G$ being a simple exceptional group of Lie type. Our main result is the following theorem.
\begin{theorem}\label{theorem1}
	Let $\D$ be a $t$-$(k^2,k,\lam)$ design, $G$ be an almost simple  block-transitive automorphism group of $\D$. Then the socle $X$ of $G$ cannot be a finite simple exceptional group of Lie type.	
\end{theorem}
     Notably, reference \cite{MR4481043} demonstrates that for a $2$-$(k^2,k,\lam)$ design $\D$ with $\lam \mid k$, if $\D$ admits a flag-transitive automorphism group $G$, then the socle of $G$ cannot be a finite simple exceptional group of Lie type. This conclusion is presented in Theorem \ref{theorem1}.
\section{Preliminary results}
    In this section, we state some useful facts in both design theory and group theory.
    For two finite groups $H$ and $K$, the notations $(H.K), (H:K)$ and $ (H\times K)$ are used to represent an extension of $H$ by $K$, a split extension (or semidirect product) of $H$ by $K$ and the direct product of $H$ and $K$, respectively.
    Let $n$ be a positive integer, then $n_p$ denotes the $p$-part of $n$, that is to say, $n_p=p^t$ if $p^t\mid n$ but $p^{t+1}\nmid n$. And $n_{p'}$ denotes the $p'$-part of $n$, which means $n_{p'}=n/n_{p}$. In addition, for two positive integers $m$ and $n$, we use $(m,n)$ to denote the greatest common divisor of $m$ and $n$. 
The first lemma is an elementary result on subgroups of almost simple groups.
\begin{lemma}\label{lemax}{\rm \cite[Lemma 2.2]{MR3483948}}
	Let $G$ be an almost simple group with socle $X$, and let $H$ be maximal in $G$ not containing $X$. Then $G=HX$, and $|H|$ divides $|\Out(X)|\cdot |H\cap X|$.
\end{lemma}
    Here, we collect and prove some useful results about $t$-$(v,k,\lambda)$ designs with $v=k^2$.
\begin{lemma}\label{le0}{\rm \cite[Corollary 2.1]{MR4830077}}
	Let $\D$ be a $t$-$(v,k,\lam)$ design with $v=k^2$, and let $n$ be a nontrivial subdegree of $G$. If $G\leq \Aut(\D)$ is block-transitive, then $k+1\mid n$.
\end{lemma}
    According to Lemma \ref{le0}, we get the following useful corollary:
\begin{corollary}\label{le2}
	Let $\D$ be a $t$-$(v,k,\lam)$ design with $v=k^2$, and let $x\in \P$. If $G\leq \Aut(\D)$ is block-transitive, then $k+1\mid(|G_{x}|, v-1)$.
\end{corollary}
\begin{remark}\label{com}
    \rm {Corollary \ref{le2} offers a quick way to rule out unfeasible candidates. By using the Magma (\cite{MR1484478}) command ``XGCD'', we can obtain the greatest common divisor of two polynomials (refer to \cite{MR4759664}).}
\end{remark}

\begin{corollary}\label{le1}
	Let $\D$ be a block-transitive $t$-$(v,k,\lam)$ design with $v=k^2$, and let $x\in \P$, then $|G|\leq(|G_{x}|-1)^2|G_{x}|$. In particular, we have $|G|<|G_{x}|^3$.
\end{corollary}
\begin{proof}
	By Corollary \ref{le2}, we get $k\leq |G_x|-1$, then $|G:G_x|=v\leq (|G_x|-1)^2$.
	Thus, $|G|\leq |G_x|(|G_x|-1)^2$, and then $|G|<|G_x|^3$.
\end{proof}
\medskip
    Let $G$ be a finite group, and $H\leq G$, then $H$ is said to be large if $|G|\leq|H|^3$ (\cite{MR3272379}). The following lemma state the classification of large maximal subgroups of almost simple groups with socle being a finite simple exceptional group of Lie type.
\begin{lemma}\label{le6}{\rm \cite[Theorem 1.2]{MR4401918}}
	Let $G$ be a finite almost simple group whose socle $X$ is a finite simple exceptional group of Lie type, and let $H$ be a maximal subgroup of $G$ not containing $X$. If $H$ is a large subgroup of $G$, then $H$ is either parabolic, or one of the subgroups listed in Table {\rm \ref{table:label1}}.
\end{lemma}
\begin{footnotesize}
	\begin{longtabu}{lll}
		\caption{Large maximal non-parabolic subgroups $H$ of almost simple groups $G$ with socle $X$ a finite simple exceptional group of Lie type}
		\label{table:label1}\\
		\hline
		$X$ & $H\cap X$ or type of $H$ & Conditions \\
		\hline
		\endfirsthead
		\hline
		$X$ & $H\cap X$ or type of $H$ & Conditions \\
		\hline
		\endhead
		\hline
		\endfoot
		$^{2}B_2(q)\ (q = 2^{2n + 1}\geqslant 8)$ & $(q + \sqrt{2q} + 1){:}4$ & $q = 8, 32$ \\
		& $^{2}B_2(q^{1/3})$ & $q>8, 3\mid 2n + 1$ \\
		$^{2}G_2(q)\ (q = 3^{2n + 1}\geqslant 27)$ & $A_1(q)$ &  \\
		& $^{2}G_2(q^{1/3})$ & $3\mid 2n + 1$ \\
		$^{3}D_4(q)$ & $A_1(q^{3})A_1(q),\ (q^{2}+\epsilon q + 1)A_2^{\epsilon}(q),\ G_2(q)$ & $\epsilon = \pm$ \\
		& $^{3}D_4(q^{1/2})$ & $q$ square \\
		& $7^2:SL_2(3)$ & $q = 2$ \\
		$^{2}F_4(q)\ (q = 2^{2n + 1}\geqslant 8)$ & $^{2}B_2(q)\wr 2,\ B_2(q):2,\ ^{2}F_4(q^{1/3})$ &  \\
		& $SU_3(q):2,\ PGU_3(q):2$ & $q = 8$ \\
		& $A_2(3):2,\ A_1(25),\ Alt_6\cdot 2^2,\ 5^2:4Alt_4$ & $q = 2$ \\
		$G_2(q)$ & $A_2^{\epsilon}(q),\ A_1(q)^2,\ G_2(q^{1/r})$ & $r = 2,3$ \\
		& $^{2}G_2(q)$ & $q = 3^a,\ a$ is odd \\
		& $G_2(2)$ & $q = 5,7$ \\
		& $A_1(13),\ J_2$ & $q = 4$ \\
		& $J_1$ & $q = 11$ \\
		& $2^3{.}A_2(2)$ & $q = 3,5$ \\
		$F_4(q)$ & $B_4(q),\ D_4(q),\ ^{3}D_4(q)$ &  \\
		& $F_4(q^{1/r})$ & $r = 2,3$ \\
		& $A_1(q)C_3(q)$ & $p\neq 2$ \\
		& $C_4(q),\ C_2(q^2),\ C_2(q)^2$ & $p = 2$ \\
		& $^{2}F_4(q)$ & $q = 2^{2n + 1}\geqslant 2$ \\
		& $^{3}D_4(2)$ & $q = 3$ \\
		& $Alt_{9},\ Alt_{10},\ A_3(3),\ J_2$ & $q = 2$ \\
		& $A_1(q)G_2(q)$ & $q>3$ odd \\
		& $Sym_6\wr Sym_2$ & $q = 2$ \\
		$E_6^{\epsilon}(q)$ & $A_1(q)A_5^{\epsilon}(q),\ F_4(q)$ &  \\
		& $(q - \epsilon)D_5^{\epsilon}(q)$ & $\epsilon = -$ \\
		& $C_4(q)$ & $p\neq 2$ \\
		& $E_6^{\pm}(q^{1/2})$ & $\epsilon = +$ \\
		& $E_6^{\epsilon}(q^{1/3})$ &  \\
		& $(q - \epsilon)^2{.}D_4(q)$ & $(\epsilon,q)\neq (+,2)$ \\
		& $(q^{2}+\epsilon q + 1){.}^3D_4(q)$ & $(\epsilon,q)\neq (-,2)$ \\
		& $J_3,\ Alt_{12},\ B_3(3),\ Fi_{22}$ & $(\epsilon,q)= (-,2)$ \\
		$E_7(q)$ & $(q - \epsilon)E_6^{\epsilon}(q),\ A_1(q)D_6(q),\ A_7^{\epsilon}(q),\ A_1(q)F_4(q),\ E_7(q^{1/r})$ & $\epsilon = \pm$ and $r = 2,3$ \\
		& $Fi_{22}$ & $q = 2$ \\
		$E_8(q)$ & $A_1(q)E_7(q),\ D_8(q),\ A_2^{\epsilon}(q)E_6^{\epsilon}(q),\ E_8(q^{1/r})$ & $\epsilon = \pm$ and $r = 2,3$ \\
		
		\hline
	\end{longtabu}
\end{footnotesize}
\begin{remark}
    \rm {For the standard notations and the orders of groups, the Atlas (\cite{MR827219}) can be consulted. Note that $^2B_2(2)$, $G_2(2)$, $^2G_2(3)$ and $^2F_4(2)$ are not simple groups (\cite[Theorem 5.1.1]{MR1057341}).
	In Table \ref{table:label1}, the type of $H$ provides an approximate description of the group-theoretic structure of $H$, for a precise structural analysis: refer to \cite{MR162840} for $^2B_2(q)$, \cite{MR955589} for $^2G_2(q)$, \cite{MR1106340} for $^2F_4(q)$, \cite{MR937609} for $^3D_4(q)$, \cite{MR618376,MR955589} for $G_2(q)$, and \cite[Table 5.1]{MR1168190} for $F_4(q)$, $E_6(q)$, $^2E_6(q)$, $E_7(q)$ and $E_8(q)$.}
\end{remark}
\begin{lemma}\label{le7}{\rm \cite[Theorem 4.1]{MR4401918}}
	Let $G$ be an almost simple group with socle $X=X(q)$ being an exceptional group of Lie type, and let $H$ be a maximal subgroup of $G$ as in Table {\rm \ref{table:label2}}. Then the action of $G$ on the cosets of $H$ has subdegrees dividing $|H:K|$, where $K$ is the subgroup of H listed in the third column of Table {\rm \ref{table:label2}}.
\end{lemma}
\begin{footnotesize}
	\begin{longtabu}{lllll}
		\caption{Some subdegrees of finite exceptional Lie type groups}
		\label{table:label2}\\
		\hline
		$X$&$H_0$&$K$&Conditions&\\
		\hline
		\endfirsthead
		\hline
		$X$&$H_0$&$K$&Conditions&\\
		\hline
		\endhead
		\hline
		\endfoot
		
		$E_8(q)$&$A_1(q)E_7(q)$&$A_1(q)A_1(q)D_6(q)$&$$\\
		$E_8(q)$&$A_1(q)E_7(q)$&$E_6^{\epsilon}(q)$&$\epsilon=\pm$\\
		$E_7(q)$&$A_1(q)D_6(q)$&$A_5^{\epsilon}(q)$&$\epsilon=\pm$\\
		$E_7(q)$&$A_7^{\epsilon}(q)$&$D_4(q)$&$q$ odd\\
		$E_7(q)$&$A_7^{\epsilon}(q)$&$C_4(q)$&$q>2$ even\\
		$E_7(q)$&$E_6^{\epsilon}(q)T_1^{\epsilon}$&$F_4(q)$&$(q,\epsilon)\neq (2,-)$\\
		$E_7(q)$&$E_6^{\epsilon}(q)T_1^{\epsilon}$&$D_5^{\epsilon}(q)$&$(q,\epsilon)\neq (2,-)$\\
		$E_6^{\epsilon}(q)$&$A_1(q)A_5^{\epsilon}(q)$&$A_2^{\epsilon}(q)A_2^{\epsilon}(q)$&$$\\
		$E_6^{\epsilon}(q)$&$A_1(q)A_5^{\epsilon}(q)$&$A_2(q^2)$&$$\\
		$E_6^{\epsilon}(q)$&$D_5^{\epsilon}T_1^{\epsilon}$&$D_4(q)$&$$\\
		$E_6^{\epsilon}(q)$&$D_5^{\epsilon}T_1^{\epsilon}$&$A_4^{\epsilon}(q)$&$(q,\epsilon)\neq (2,-)$\\
		$E_6^{\epsilon}(q)$&$C_4$&$C_2(q)C_2(q)$&$q$ odd\\
		$E_6^{\epsilon}(q)$&$C_4$&$A_3^{\delta}(q)$&$q$ odd, $\delta=\pm$\\
		$F_4(q)$&$D_4(q)$&$G_2(q)$&$q>2$\\
		$F_4(q)$&$D_4(q)$&$A_3^{\epsilon}(q)$&$\epsilon=\pm$, $(q,\epsilon)\neq (2,-)$\\
		$F_4(q)$&$^3D_4(q)$&$G_2(q)$&$q>2$\\
		$F_4(q)$&$^3D_4(q)$&$A_2^{\epsilon}(q)$&$\epsilon=\pm$\\
		
		\hline
	\end{longtabu}
\end{footnotesize}
\begin{remark}
	\rm { Note that $H_0$ denotes a normal subgroup of very small index in $H$ for notational convenience. For the precise structures of $H$ and $K$ in Table \ref{table:label2}, please refer to \cite[Remark 4.2]{MR4401918}. }
\end{remark}
    The following are some helpful conclusions for simple groups of Lie type.
\begin{lemma}\label{le3}{\rm \cite[1.6]{MR340446}}
	{\rm (Tits Lemma)} If $X$ is a simple group of Lie type in characteristic $p$, then any proper subgroup of index prime to p is contained in a proper parabolic subgroup of $X$.
\end{lemma}
\begin{corollary}\label{le4}
	Let $\D$ be a block-transitive $t$-$(v,k,\lam)$ design with $v=k^2$, $G\leq \Aut(\D)$, $x\in \P$, and $X=\Soc(G)$ be a simple group of Lie type in characteristic $p$. If the point stabilizer $G_x$ is not a parabolic subgroup of $G$, then $(p,\ k+1)=1$ and $|G|<|G_x||G_x|^2_{p^{\prime}}$.
\end{corollary}
\begin{proof}
    By Lemma \ref{le3}, we have $p\mid v$, otherwise, $(p,|G:G_x|)=1$, which implies that $X_x$ is contained in a parabolic subgroup of $X$ for $|G:G_x|=|X:X_x|$. Then $(p,\ v-1)=1$. Since $k+1\mid v-1$, then $(k+1,\ p)=1$. Thus $k+1\mid |G_x|_{p'}$ for $k+1\mid |G_x|$. Similar to the proof of Corollary \ref{le1}, we can obtain $|G|<|G_x||G_x|^2_{p'}$.
\end{proof}
\begin{lemma}\label{le5}{\rm \cite{MR907231}}
	If $X$ is a simple group of Lie type in characteristic $p$, acting on the set of cosets of a maximal parabolic subgroup, and $X$ is not $L_{d}(q)$, $P\Omega^+_{2m}(q)$ (with $m$ odd) and $E_6(q)$, then there is a unique subdegree which is a power of $p$.
\end{lemma}
\begin{remark}\label{le5re}
	\rm The proof of {\cite[Lemma 2.6]{MR1940339} implies that for an almost simple group $G$ with socle $X=E_6(q)$, if $G$ contains a graph automorphism or the maximal parabolic subgroups $H=P_i$ with one of $2$ and $4$, then the conclusion of Lemma \ref{le5} is still applies (refer to \cite{MR4401918}). }  
\end{remark}
\section{Proof of Theorem}
	In this section, we will prove Theorem \ref{theorem1} by a series of lemmas. Throughout of the rest of the paper, we assume the following:
	
    \textbf{Hypothesis:} Let $\D=(\P,\B)$ be a $t$-$(v,k,\lam)$ design with $v=k^2$, $G\leq \Aut(\D)$ be block-transitive, $x\in \P$, and $X=\Soc(G)$ be a finite simple exceptional group of Lie type.

    Recall that $G$ is point-primitive, then the point stabilizer $G_x$ is maximal in $G$ not containing $X$. Then $G_x$ is either parabolic or one of the subgroups listed in Table \ref{table:label1} according to Corollary \ref{le1} and Lemma \ref{le6}. Note that $v=|G|/|G_x|=|X|/|G_x\cap X|$, and $|G_x|=f|G_x\cap X|$ if $|G|=f|X|$, where $f\mid |\Out(X)|$.
We first consider the parabolic cases through the Lemma \ref{parabolic}. The definition and structure of parabolic subgroup are detailed in \cite[Definition 2.6.4 and Theorem 2.6.5]{MR1490581}, and the notation $P_i$ denotes a standard maximal parabolic subgroup corresponding to deleting the $i$-th node in the Dynkin diagram of $X$. 

\begin{lemma}\label{parabolic}
	The group $G_x\cap X$ is not a parabolic subgroup of $X$.
\end{lemma}
\begin{proof}
If $X={^2B_2(q)}$ with $q=2^{2e+1}$, then $G_x \cap X=[q^2]:Z_{q-1}$. Thus $|G_x \cap X|=q^{2}(q-1)$ and $v=|G:G_x|=q^{2}+1$. We get $k+1 \mid q^{2}$ for $k+1 \mid v-1$, and then there exists a positive integer $m$ such that $k=2^m-1$, this implies that $(2^m-1)^2=2^{2(2e+1)}+1$, which is impossible.

If $X={^2G_2(q)}$ with $q=3^{2e+1}$, then $G_x \cap X=[q^3]:Z_{q-1}$. Thus $|G_x \cap X|=q^3(q-1)$ and $v=q^3+1$. We get $k+1\mid q^3$ for $k+1 \mid v-1$, and then there exists a positive integer $m$ such that $k=3^{m}-1$, this implies that $(3^{m}-1)^2=3^{3(2e+1)}+1$, which is impossible.

If $X=E_6(q)$ with $q=p^e$, then $|G_x|=f|G_x\cap X|$, where $f\mid 2de$ and $d=(3,q-1)$. Suppose $G_x \cap X=P_1$ or $P_3$. If $G_x \cap X=P_1$, by \cite{MR1940339}, we have $$v=(q^{12}-1)(q^9-1)/(q^4-1)(q-1),$$ and the non-trivial subdegrees are $$n_1=q(q^7+q^6+q^5+q^4+q^3+q^2+q+1)(q^3+1),$$
and
$$n_2=q^8(q^4+q^3+q^2+q+1)(q^4+1).$$ Then, $(n_1,\ n_2)\mid q(q^4+1)$. Hence, $k+1\mid q(q^4+1)$ by Lemma \ref{le0}. Then we get  
$$(q^{12}-1)(q^9-1)<q^2(q^4+1)^2(q^4-1)(q-1)$$
for $v=k^2$, a contradiction.
If $G_x \cap X=P_3$, by \cite{MR1940339}, we have $$|G_x \cap X|=\frac{1}{d} q^{36}(q-1)^{6}(q+1)^3(q^2+1)(q^2+q+1)(q^4+q^3+q^2+q+1),$$
and  $$v=(q^3+1)(q^4+1)(q^6+1)(q^4+q^2+1)(q^8+q^7+q^6+q^5+q^4+q^3+q^2+q+1).$$
Clearly, $|v-1|_q=q$, $(v-1,(q+1)(q^2+1)(q^2+q+1))=1$, and $q^4+q^3+q^2+q+1\mid v-1$.
Thus, $(v-1,|G_x|)\mid fq(q-1)^6(q^4+q^3+q^2+q+1)$. From Corollary \ref{le2}, it follows that $k+1$ divides $fq(q-1)^6(q^4+q^3+q^2+q+1)$. Thus we get $$v=k^2<(2de)^2q^2(q-1)^{12}(q^4+q^3+q^2+q+1)^2,$$ which is impossible.

If $(X,G_x\cap X)\neq (E_6(q),P_i)$ with $i=1$ and $3$, and $X\neq {^2B_2(q)}$, $^2G_2(q)$. According to Lemma \ref{le5} and Remark \ref{le5re}, there is a unique subdegree which is a power of $p$. Thus, $k+1\mid |v-1|_{p}$. The values of $|v-1|_{p}$ for each case have been determined in \cite[Table 5]{2017arXiv170201257H}. 
Then we can obtain the upper bounds of all $k$ for $k<|v-1|_{p}$, which is too small to satisfy $v=k^2$.
\end{proof}
\medskip
Next, we deal with some numerical cases by Lemma \ref{easy}.

\begin{lemma}\label{easy}
	If $X$ and $G_x\cap X$ are as in Table {\rm{\ref{table:label4}}}, then the hypothesis does not hold.
\end{lemma}
\begin{footnotesize}
	\begin{longtabu}{lllll}
		\caption{Parameter $v$ in some numerical cases}
		\label{table:label4}\\
		\hline
		$X$&$G_x\cap X$&$|G_x\cap X|$&$v$\\
		\hline
		\endfirsthead
		\hline
		$X$&$G_x\cap X$&$|G_x\cap X|$&$v$\\
		\hline
		\endhead
		\hline
		\endfoot
		
		$^2B_2(8)$&$13:4$&$52$&$560$\\
		$^2B_2(32)$&$41:4$&$164$&$198400$\\
		
		$^3D_4(2)$&$7^2:SL_2(3)$&$1176$&$179712$\\
		
		$^2F_4(8)$&$SU_3(8):2$&$33094656$&$8004475184742400$\\
		&$PGU_3(8):2$&$33094656$&$8004475184742400$\\
		
		$G_2(3)$&$2^3\cdot A_2(2)$&$1344$&$3159$ \\
		$G_2(4)$&$A_1(13)$&$1092$&$230400$\\
		&$J_2$&$604800$&$416$\\
		$G_2(5)$&$2^3\cdot A_2(2)$&$1344$&$4359375$ \\
		&$G_2(2)$&$12096$&$484375$\\
		$G_2(7)$&$G_2(2)$&$12096$&$54925276$\\
		$G_2(11)$&$J_1$&$175560$&$2145199320$\\
		
		$F_4(2)$&$Alt_9$&$181440$&$18249154560$ \\
		&$Alt_{10}$&$1814400$&$1824915456$ \\
		&$A_3(3)\cdot 2$&$12130560$&$272957440$ \\
		&$J_2$&$604800$&$5474746368$ \\
		& $(Sym_6\wr Sym_2)\cdot 2$&$2880$&$1149696737280$ \\
		$F_4(3)$&$^3D_4(2)$&$211341312$&$27133458851701800$ \\
		
		$E_6^{-}(2)$&$J_3$&$50232960$&$1523551064555520$ \\
		&$Alt_{12}$&$239500800$&$319549996007424$ \\
		&$B_3(3):2$&$9170703360$&$8345322782720$ \\
		&$Fi_{22}$&$64561751654400$&$1185415168$ \\
		
		$E_7(2)$&$Fi_{22}$&$64561751654400$&$123873281581429293827751936$ \\
		\hline
	\end{longtabu}
\end{footnotesize}
\begin{proof}
    If $X=G_2(4)$ and $G_x\cap X=A_1(13)$, then $k=480$ for $v=230400$, which does not hold for $k+1\mid v-1$. For other cases in Table \ref{table:label4}, the values of $v$ cannot be square-rooted, which can be ruled out.
\end{proof}
\medskip
The following lemma deals with the subfield cases. The definition of subfield subgroup is detailed in \cite[Definition 2.2.11]{MR3098485}. Note that the subfield subgroups of $E_6(q)$ are $E_6(q_0)$ if $q=q_0^r$ with $r$ prime, and also ${^2E_6(q^{1/2})}$ if $q$ is a square (refer to \cite[p. 638]{MR4651009}).
\begin{lemma}\label{subfield}
	The group $G_x \cap X$ is not a subfield subgroup of $X$.
\end{lemma}
\begin{footnotesize}
	\begin{longtabu}{lll}
		\caption{The possible values of $q$ in subfield cases}
		\label{table:label3}\\
		\hline
		$X$&$G_x \cap X$&Possible $q$\\
		\hline
		\endfirsthead
		\hline
		$X$&$G_x \cap X$&Possible $q$\\
		\hline
		\endhead
		\hline
		\endfoot
		
		$^2B_2(q)$&$^2B_2(q^{1/3})$&None\\
		$^2G_2(q)$&$^2G_2(q^{1/3})$&None\\
		$^3D_4(q)$&$^3D_4(q^{1/2})$&None\\
		$^2F_4(q)$&$^2F_4(q^{1/3})$&None\\
		$G_2(q)$&$G_2(q^{1/2})$&None\\
		$G_2(q)$&$G_2(q^{1/3})$&None\\
		$F_4(q)$&$F_4(q^{1/2})$&None\\
		$F_4(q)$&$F_4(q^{1/3})$&None\\
		$E_6(q)$&$E_6(q^{1/2})$&$2^2$\\
		$E_6(q)$&$E_6(q^{1/3})$&$2^3$\\
		$E_6(q)$&$^2E_6(q^{1/2})$&$2^2$\\
		$^2E_6(q)$&$^2E_6(q^{1/3})$&$2^3$\\
		$E_7(q)$&$E_7(q^{1/2})$&$i^2$ with $i\leq5$\\
		$E_7(q)$&$E_7(q^{1/3})$&None\\
		$E_8(q)$&$E_8(q^{1/2})$&$2^2$\\
		$E_8(q)$&$E_8(q^{1/3})$&None\\
		
		\hline
	\end{longtabu}
\end{footnotesize}
\begin{proof}             
    Suppose $(X,G_x\cap X)=({^2B_2(q)},{^2B_2(q_0)})$ with $q=p^e$ and $q=q_0^3$, then $|G_x|=f|G_x\cap X|$ where $f\mid e$, $G_x\cap X=q^2_0(q^2_0+1)(q_0-1)$, and $$v=\frac{q_0^{6}(q_0^{6}+1)(q_0^3-1)}{q^2_0(q^2_0+1)(q_0-1)}=q_0^4(q_0^4-q_0^2+1)(q_0^2+q_0+1).$$
    By using Magma (Remark \ref{com}), there are two integral coefficient polynomials $P(q_0)$ and $Q(q_0)$ on $q_0$, such that
    $$P(q_0)|G_x\cap X|+Q(q_0)(v-1)=20.$$
    Thus, $(|G_x|,v-1)\mid 20f$. Hence $k+1\mid 20f$. Since $v=k^2<(|G_x|,v-1)^2$, we get $$q_0^4(q_0^5+q_0^4+q_0^3+q_0^2+q_0+1)(q_0^4-q_0^2+1)<20^2f^2.$$
    As a result, it can be demonstrated that no value of $q$ satisfies the inequality $v<(|G_x|,v-1)^2$. 
    
    Using similar computations, we determine the values of $q$ in all subfield cases that satisfy the inequality $v<(|G_x|,v-1)^2$, as listed in Table \ref{table:label3}. Notably, for each $q$ value listed in Table \ref{table:label3}, the corresponding $v$ fail to be a perfect square. Thus, the conclusion is proved.
\end{proof}
\medskip
Based on Lemmas \ref{parabolic}-\ref{subfield}, we can obtain the following conclusions through checking the remaining candidates in Table \ref{table:label1}.	 	
\begin{lemma}\label{^{2}B_{2}(q)}
	The group $X$ is not $^{2}B_{2}(q)$, with $q=2^{2e+1}\geq 8$.
\end{lemma}
\begin{proof}
    According to Lemma \ref{parabolic}, $G_x\cap X$ cannot be parabolic subgroup. And $(X,G_x\cap X)$ cannot be $({^2B_2(8)},13:4)$, $(^2B_2(32),41:4)$ or $({^2B_2(q)},{^2B_2(q^{1/3})})$ by Lemmas \ref{easy} and \ref{subfield}. Thus, ${^{2}B_{2}(q)}$ can be ruled out.
\end{proof}
\begin{lemma}
	The group $X$ is not $^{2}G_{2}(q)$, with $q=3^{2e+1}\geq27$.
\end{lemma}
\begin{proof}
    Suppose that $X={^{2}G_{2}(q)}$ with $q=3^{2e+1}\geq27$. Then $$|G|=f|X|=fq^3(q^3+1)(q-1),$$ where $f\mid2e+1$. It is obvious that $f<\sqrt{q}$. According to Lemmas \ref{parabolic}-\ref{subfield}, we only need to deal with the case $G_x \cap X=2\times L_{2}(q)$ with $q\geq 27$. Then $|G_x\cap X|=q(q^2-1)$ and $v=q^2(q^2-q+1)$. Thus, $(|G_x|,v-1)\mid f(q-1)$. Therefore, $k+1\mid f(q-1)$. Hence, $q^2(q^2-q+1)<q(q-1)^2$ for $v=k^2$. It implies that $q^3-2q^2+3q-1<0$, which is a contradiction.
\end{proof}
\begin{lemma}\label{^{3}D_{4}(q)}
	The group $X$ is not $^{3}D_{4}(q)$, with $q=p^{e}$.
\end{lemma}
\begin{proof}
    Suppose that $X={^{3}D_{4}(q)}$ with $q=p^{e}$, where $p$ is a prime. Then $$|G|=f|X|=fq^{12}(q^8+q^4+1)(q^6-1)(q^2-1),$$ where $f\mid 3e$.  It is obvious that $f<q^{\frac 32}$. According to Lemmas \ref{parabolic}-\ref{subfield}, the remaining candidates for $G_x\cap X$ are: $A_1(q^3)\times A_1(q)$ with $q$ even, $(SL_2(q^3)\circ SL_2(q))\cdot 2$ with $q$ odd, $(Z_{q^2+\epsilon q+1}\circ SL^{\epsilon}_3(q))\cdot d\cdot 2$ with $d=(3,q^2+\epsilon q+1)$ and $\epsilon=\pm$, and $G_2(q)$.

    Case (1).\ If $G_x \cap X=A_1(q^3)\times A_1(q)$ with $q$ even or $(SL_2(q^3)\circ SL_2(q))\cdot 2$ with $q$ odd, then $$|G_x \cap X|=q^{4}(q^6-1)(q^2-1)=q^{4}(q^4+q^2+1)(q^2-1)^2,$$
    and
    $$v=q^8(q^8+q^4+1)=q^8(q^4+q^2+1)(q^4-q^2+1).$$ Clearly, $(v-1,q(q^4+q^2+1))=1$, and $v-1\equiv 2\pmod {q^2-1}$. Then, $(|G_x|,v-1)\mid 2^2f$. Thus, we get $k+1\mid 2^2f$, which implies $v=k^2<2^4f^2$, this is impossible.

    Case (2).\ If $G_x \cap X=(Z_{q^2+\epsilon q+1}\circ  SL^{\epsilon}_3(q))\cdot d\cdot 2$ with $d=(3,q^2+\epsilon q+1)$ and $\epsilon=\pm$, then $$|G_x\cap X|=2dq^{3}(q^3-\epsilon1)(q^2+\epsilon q+1)(q^2-1),$$
    and
    $$v=\frac{1}{2d}q^9(q^2-\epsilon q+1)(q^3+\epsilon1)(q^4-q^2+1).$$
    
    If $(\epsilon,d)=(\pm,1)$, then by using Magma (Remark \ref{com}), $(|G_x|,v-1)\mid 2f\cdot 2\cdot 3^5(q-\epsilon1)$. Hence, $k+1\mid 2^23^5f(q-\epsilon1)$. It implies that $v=k^2<2^43^{10}f^2(q-\epsilon1)^2$, which force $(\epsilon,q)=(+,2)$. In this case $v=2^8\cdot 3^3\cdot 13$, a contradiction.
    
    If $(\epsilon,d)=(\pm,3)$, then by using Magma (Remark \ref{com}), $(|G_x|,v-1)\mid 6f\cdot 2^5\cdot 3$. Hence, $k+1\mid 2^63^2f$. It implies that $v=k^2<2^{12}3^4f^2$, which force $(\epsilon,q)=(-,2)$. In this case $v=\frac{163072}{3}$, a contradiction.
   
    Case (3).\ If $G_x \cap X=G_2(q)$, then $$|G_x\cap X|=q^{6}(q^6-1)(q^2-1)=q^{6}(q^4+q^2+1)(q^2-1)^2,$$
    and $$v=q^6(q^4+q^2+1)(q^4-q^2+1).$$ Clearly, $(v-1,q(q^4+q^2+1))=1$, and $v-1\equiv 2\pmod{q^2-1}$. Thus, $(|G_x|,v-1)\mid 2^2f$, which is a similar contradicition.
\end{proof}
\begin{lemma}
	The group $X$ is not $^{2}F_{4}(q)$, with $q=2^{2e+1}\geq 8$.
\end{lemma}
\begin{proof}
	Suppose that $X={^{2}F_{4}(q)}$ with $q=2^{2e+1}\geq 8$. Then $$|G|=f|X|=fq^{12}(q^6+1)(q^4-1)(q^3+1)(q-1),$$ where $f\mid2e+1$. It is obvious that $f<q$. According to Lemmas \ref{parabolic}-\ref{subfield}, we only need to deal with the cases $G_x \cap X={^{2}B_2(q)}\wr2$ and $B_2(q):2$.
	
	Case (1).\ If $G_x \cap X={^{2}B_2(q)}\wr2$, then $|G_x \cap X|=2q^{4}(q^2+1)^2(q-1)^2$, $v=\frac12q^8(q^4-q^2+1)(q^3+1)(q+1)$. Clearly,
	$(v-1,q(q-1))=1$, and $v-1\equiv 4 \pmod{q^2+1}$. Thus, $k+1\mid 2^5f$ for $(|G_x|,v-1)\mid 2^5f$. Therefore, $v=k^2<2^{10}f^2<2^{10}q^2$, a contradiction.
	
	Case (2).\ If $G_x \cap X=B_2(q):2$, then $$|G_x \cap X|=2q^{4}(q^4-1)(q^2-1)=2q^{4}(q^2+1)(q+1)^2(q-1)^2,$$ and $v=\frac12q^8(q^6+1)(q^2-q+1)$. Clearly, $(v-1,q(q^2+1))=1$, $v-1\equiv 2 \pmod{q+1}$, and $q-1\mid v-1$. Thus, $k+1\mid 2^3f(q-1)^2$ for $(|G_x|,v-1)\mid 2^3f(q-1)^2$. Therefore, $v=k^2<2^6f^2(q-1)^4<2^6q^2(q-1)^4$, a contradiction.
\end{proof}
\begin{lemma}
	The group $X$ is not $G_{2}(q)$, with $q=p^{e}>2$.
\end{lemma}
\begin{proof}
    Suppose that $X=G_{2}(q)$ with $q=p^{e}>2$. Then $|G|=f|X|=fq^{6}(q^6-1)(q^2-1)$, where $f\mid 2e$ for $p=3$, and $f\mid e$ for $p\neq 3$. It is obvious that $f<q$. According to Lemmas \ref{parabolic}-\ref{subfield}, the remaining candidates for $G_x\cap X$ are: $SL^{\epsilon}_{3}(q)\cdot 2$ with $\epsilon=\pm$, $d\cdot A_1(q)^2:d$ with $d=(2,q-1)$, and ${^2G_2(q)}$ with $q=3^{a}\geq 27$ and $a$ odd.

    Case (1).\ If $G_x \cap X=SL^{\epsilon}_{3}(q)\cdot 2$, then $v=\frac{q^3(q^3+\epsilon1)}{2}$. Note that the factorization $\Omega_7(q)=G_2(q)N^{\epsilon}_1$ (\cite{MR1016353}), and the $\Omega_7(q)$-suborbits are unions of $G_2(q)$-suborbits (\cite{MR928521}).
    
    If $q$ is odd, the $\Omega_7(q)$-subdegrees are $(q^3-\epsilon1)(q^3+\epsilon1)$, $\frac{1}{2}q^2(q^3-\epsilon1)$ and $\frac{1}{2}q^2(q^3-\epsilon1)(q-3)$ (\cite[Section 4, case 1]{MR1940339}). Thus, $k+1$ divides $\frac{1}{2}(q^3-\epsilon1)$ by Lemma \ref{le0}. Let $k+1=\frac{1}{2m}(q^3-\epsilon1)$, here $m$ is a positive integer. Then from $(k+1)(k-1)=v-1$, we get $$\frac{1}{2m}(q^3-\epsilon1)\cdot (\frac{1}{2m}(q^3-\epsilon1)-2)=\frac{q^3(q^3+\epsilon1)}{2}-1=\frac{(q^3-\epsilon1)(q^3+2\epsilon)}{2}.$$
    Therefore,
    $$\frac{1}{m}\cdot (\frac{1}{2m}(q^3-\epsilon1)-2)=q^3+2\epsilon.$$
    Thus,
    $$q^3=\frac{4m+3\epsilon}{1-2m^2}-2\epsilon,$$
    which forces $(\epsilon,m)=(+,m)$, then $q^3=1$, a contradiction.

    If $q$ is even, then $p=2$ and $\Omega_7(q)\cong Sp_6(q)$. The $Sp_6(q)$-subdegrees are $(q^3-\epsilon1)(q^2+\epsilon1)$ and $\frac{1}{2}q^2(q^3-\epsilon1)(q-2)$ (\cite[Section 3, case 8]{MR1940339}). Then $k+1\mid q^3-\epsilon1$. Let $k+1=\frac{1}{m}(q^3-\epsilon1)$, here $m$ is a positive integer. Similarly, we get $$\frac{1}{m}\cdot (\frac{1}{m}(q^3-\epsilon1)-2)=\frac{1}{2} (q^3+2\epsilon).$$
    Thus,
    $$q^3=\frac{4m+6\epsilon}{2-{m}^2}-2\epsilon,$$
    which forces $(\epsilon,m)=(+,1)$ or $(-,2)$, then $q^3=8$ or $1$ respectively, a contradiction.

    Case (2).\ If $G_x \cap X=d\cdot A_1(q)^2:d$ with $d=(2,q-1)$, then $|G_x\cap X|= q^2(q^2-1)^2$ and $v=q^4(q^4+q^2+1)$. Clearly, $(v-1,q)=1$, and $v-1\equiv 2 \pmod{q^2-1}$. Thus, $(|G_x|,v-1)\mid 2^2f$. Hence, $k+1\mid 2^2f$. Then we get $q^4(q^4+q^2+1)<2^{4}f^2$ for $v=k^2$, which is impossible.

    Case (3).\ If $G_x \cap X={^2G_2(q)}$ with $q=3^{a}\geq 27$ and $a$ odd, then $$|G_x\cap X|=q^3(q^3+1)(q-1)=q^3(q^2-q+1)(q^2-1),$$ $v=q^3(q^3-1)(q+1)$. It is easy to know $(v-1,q(q^2-1))=1$, and $(v-1,q^2-q+1)\mid 7$. Thus, $(|G_x|,v-1)\mid 7f$. Hence, $k+1\mid 7f$. Then we get $q^3(q^3-1)(q+1)<7^2f^2$ for $v=k^2$, which is impossible.
\end{proof}
\begin{lemma}\label{F_{4}(q)}
	The group $X$ is not $F_{4}(q)$, with $q=p^{e}$.
\end{lemma}
\begin{proof}
    Suppose that $X=F_{4}(q)$ with $q=p^{e}$, where $p$ is a prime. Then $$|G|=f|X|=fq^{24}(q^{12}-1)(q^8-1)(q^6-1)(q^2-1),$$ where $f\mid (2,p)e$. It is obvious that $f\leq q$ for $p=2$, and $f\leq \sqrt{q}$ for $p\neq 2$. According to Lemmas \ref{parabolic}-\ref{subfield}, the remaining candidates for $G_x\cap X$ are: $2\cdot \Omega_9(q)$, $(2,q-1)^2\cdot D_4(q)\cdot S_3$, ${^3D_4(q)}\cdot 3$, $2\cdot (A_1(q)\times C_3(q))\cdot 2$ with $p\neq 2$, $C_4(q)$ with $p=2$, $Sp_4(q^2)\cdot 2$ with $p=2$, $Sp_4(q)^2\cdot 2$ with $p=2$, ${^2F_4(q)}$ with $q=2^{a}\geq 2$ and $a$ odd, and $\Soc(G_x)=A_1(q)\times G_2(q)$ with $q>3$ odd.

    Case (1).\ If $G_x \cap X=2\cdot \Omega_9(q)$, then
    \begin{align*}
    |G_x \cap X|&=q^{16}(q^8-1)(q^6-1)(q^4-1)(q^2-1)\\
    &=q^{16}(q^4+1)(q^2+q+1)(q^2-q+1)(q^2+1)^2(q^2-1)^4,
    \end{align*}
    and $v=q^8(q^8+q^4+1).$
    Clearly,
    \begin{align*}
    	v-1&\equiv 1 \pmod{q},&v-1&\equiv 0 \pmod{q^4+1},\\
    	v-1&\equiv 1 \pmod{q^2+q+1}, &v-1&\equiv 1 \pmod{q^2-q+1},\\
    	v-1&\equiv 2 \pmod{q^2+1}, &v-1&\equiv 2 \pmod{q^2-1}.
    \end{align*}
    Thus, $(|G_x|,v-1)\mid f\cdot 2^{6}\cdot(q^4+1)$. Hence $k+1\mid 2^{6}f(q^4+1)$. Then we get $$q^8(q^8+q^4+1)<2^{12}f^2(q^4+1)^2$$ for $v=k^2$, which implies $q=2$, $3$ or $4$. If $q=2$, then $v=2^8\cdot 3\cdot 7\cdot 13$, which is impossible. Similarly, $v$ is not a perfect square for other cases.

    Case (2).\ Suppose $G_x \cap X=(2,q-1)^2\cdot D_4(q)\cdot S_3$. Then
    $|G_x \cap X|=6q^{12}(q^6-1)(q^4-1)^2(q^2-1)$
    and $v=\frac{1}{6}q^{12}(q^8+q^4+1)(q^4+1).$       
    If $q=2$, then $v=2^{11}\cdot 7\cdot 13\cdot 17$, a contradiction.      
    If $q\neq 2$,     
    then according to Lemma \ref{le7}, there are two subdegrees dividing $$n_1=6fq^6(q^4-1)^2, \ n_2=6(4,q-\epsilon1)fq^6(q^4-1)(q^3+\epsilon1),$$
    respectively.
    Thus, $(n_1,n_2)\mid 6(4,q-\epsilon1)fq^6(q^4-1)$. Hence, $k+1\mid 2^33f(q^4-1)$ for $(k+1,q)=1$. Then $q^{12}(q^8+q^4+1)(q^4+1)<2^73^3f^2(q^4-1)^2$ for $v=k^2$, which implies a contradiction.

    Case (3).\ Suppose $G_x \cap X={^3D_4(q)}\cdot 3$. Then 
    $$|G_x \cap X|=3q^{12}(q^8+q^4+1)(q^6-1)(q^2-1),$$
    and $v=\frac{1}{3}q^{12}(q^8-1)(q^4-1).$ 
    If $q=2$, then $v=2^{12}\cdot 3\cdot 5^2\cdot 17$, a contradiction.
    If $q\neq 2$, then
    according to Lemma \ref{le7}, there are two subdegrees dividing $$n_1=3fq^6(q^8+q^4+1), \ n_2=3(3,q-\epsilon1)fq^9(q^8+q^4+1)(q^3+\epsilon1),$$
    respectively.
    Thus, $(n_1,n_2)\mid 3fq^6(q^8+q^4+1)$. Hence, $k+1\mid 3f(q^8+q^4+1)$ for $(k+1,q)=1$. Then we get $$q^{12}(q^8-1)(q^4-1)<3^3f^2(q^8+q^4+1)^2$$ for $v=k^2$,  a contradiction.

    Case (4).\ If $G_x \cap X=2\cdot (A_1(q)\times C_3(q))\cdot 2$ with $p\neq 2$, then $$|G_x \cap X|=q^{10}(q^6-1)(q^4-1)(q^2-1)^2=q^{10}(q^2+q+1)(q^2-q+1)(q^2+1)(q^2-1)^4,$$ and
    \begin{align*}   
    v&=q^{14}(q^{10}+q^8+q^6+q^4+q^2+1)(q^4+1)\\
    &=q^{14}(q^4+1)(q^4-q^2+1)(q^2+q+1)(q^2-q+1)(q^2+1).
    \end{align*}
    Clearly, $$(v-1,q(q^2+q+1)(q^2-q+1)(q^2+1))=1.$$
    Thus, $(|G_x|,v-1)\mid f(q^2-1)^4$. Hence, $k+1\mid f(q^2-1)^4$. Then we get $q^{28}<f^2(q^2-1)^8$
    for $v=k^2$, which is impossible.

    Case (5).\ If $G_x \cap X=C_4(q)$ with $p=2$, then $|G_x \cap X|=q^{16}(q^8-1)(q^6-1)(q^4-1)(q^2-1)^2$ and $v=q^{8}(q^8+q^4+1).$ Similar to case (1), we can obtain a contradiction.

    Case (6).\ If $G_x \cap X=Sp_4(q^2)\cdot 2$ with $p=2$, then $$|G_x \cap X|=2q^{8}(q^8-1)(q^4-1)=2q^{8}(q^4+1)(q^2+1)^2(q^2-1)^2,$$ and $$v=\frac{1}{2}q^{16}(q^8+q^4+1)(q^6-1)(q^2-1).$$     
    Clearly,
    $(v-1,q(q^2-1))=1$, and $v-1\equiv 5 \pmod{q^2+1}$. Thus, $(|G_x|,v-1)\mid 2f\cdot 5^2(q^4+1)$. Hence, $k+1\mid 2\cdot 5^2f(q^4+1)$. Then we get $$q^{16}(q^8+q^4+1)(q^6-1)(q^2-1)<2^35^4f^2(q^4+1)^2$$ for $v=k^2$, which is a contradiction.

    Case (7).\ If $G_x \cap X=Sp_4(q)^2\cdot 2$ with $p=2$, then $$|G_x \cap X|=2q^{8}(q^4-1)^2(q^2-1)^2=2q^{8}(q^2+1)^2(q^2-1)^4,$$ and $$v=\frac{1}{2}q^{16}(q^8+q^4+1)(q^4+q^2+1)(q^4+1).$$    
    Clearly,
    $(v-1,q)=1$, and $v-1\equiv 2 \pmod{q^2+1}$. Thus, $k+1\mid 2^{3}f(q^2-1)^4$ for $(|G_x|,v-1)\mid 2^{3}f(q^2-1)^4$. Then we get $q^{32}<2^{7}f^2(q^2-1)^8$ for $v=k^2$, a contradiction.

    Case (8).\ If $G_x \cap X={^2F_4(q)}$ with $q=2^{a}\geq 2$ and $a$ odd, then
    \begin{align*}
    |G_x \cap X|&=q^{12}(q^6+1)(q^4-1)(q^3+1)(q-1)\\
    &=q^{12}(q^4-q^2+1)(q^2-q+1)(q^2+1)^2(q^2-1)^2,
    \end{align*}
    and $$v=q^{12}(q^6-1)(q^4+1)(q^3-1)(q+1).$$    
    Clearly, $(v-1,q(q^2-q+1)(q^2-1))=1$.
    Thus, $(|G_x|,v-1)\mid f(q^4-q^2+1)(q^2+1)^2$. Hence, $k+1\mid f(q^4-q^2+1)(q^2+1)^2$. Then we get $q^{26}<f^2(q^4-q^2+1)^2(q^2+1)^4$ for $v=k^2$, which implies a contradiction.
    
    Case (9).\ If $\Soc(G_x)=A_1(q)\times G_2(q)$ with $q>3$ odd, then $$|G_x|=c\cdot q(q^2-1)\cdot q^6(q^6-1)(q^2-1)=cq^{7}(q^6-1)(q^2-1)^2,$$ and $$v=\frac{f}{c}q^{17}(q^{10}+q^8+q^6+q^4+q^2+1)(q^8-1),$$
    where $c\mid 4e^2$. By Corollary \ref{le4}, we get $v<|G_x|^2_{p^{\prime}}$.
    Thus, $$v<c^2(q^6-1)^2(q^2-1)^4<2^4e^4q^{20},$$     
    which is a contradiction.
\end{proof}    
\begin{lemma}\label{E_{6}(q)}
	The group $X$ is not $E^{\epsilon}_{6}(q)$, with $q=p^{e}$.
\end{lemma}
\begin{proof}
    Suppose that $X$ = $E_{6}^{\epsilon}(q)$ with $q=p^{e}$, where $p$ is a prime. Then $$|G|=f|X|=\frac{f}{d}q^{36}(q^{12}-1)(q^9-\epsilon1)(q^8-1)(q^6-1)(q^5-\epsilon1)(q^2-1),$$ where $f\mid 2de$, $d=(3,q-\epsilon1)$ and $\epsilon=\pm$. It is obvious that $f<q^2$. According to Lemmas \ref{parabolic}-\ref{subfield}, the remaining candidates for $G_x\cap X$ are: $(2,q-1)\cdot (A_1(q)\times A^{\epsilon}_5(q))\cdot (2,q-1)\cdot d$, $\Soc(G_x)=F_4(q)$, $(4,q-\epsilon1)\cdot (D_5^{\epsilon}(q)\times (\frac{q-\epsilon1}{(4,q-\epsilon1)}))\cdot (4,q-\epsilon1)$, $\Soc(G_x)=C_4(q)$ with $p$ odd, $(2,q-1)^2\cdot (D_4(q)\times (\frac{q-\epsilon1}{(2,q-1)})^2)\cdot (2,q-1)^2\cdot S_3$ with $(q,\epsilon)\neq(2,+)$, and $({^3D_4(q)}\times (q^2+\epsilon q+1))\cdot 3$ with $(q,\epsilon)\neq(2,-)$.

    Case (1).\ If $G_x \cap X=(2,q-1)\cdot (A_1(q)\times A^{\epsilon}_5(q))\cdot (2,q-1)\cdot d$, then $$|G_x \cap X|=q^{16}(q^{6}-1)(q^5-\epsilon1)(q^4-1)(q^3-\epsilon1)(q^2-1)^2,$$ and $$v=\frac{1}{d}q^{20}(q^{10}+q^8+q^6+q^4+q^2+1)(q^6+\epsilon q^3+1)(q^4+1).$$
    According to Lemma \ref{le7}, there are two subdegrees dividing $$n_1=d^2fq^{10}(q^5-\epsilon1)(q^4-1)(q^3+\epsilon1), \ n_2=(3,q^2-1)fq^{10}(q^5-\epsilon1)(q^3-\epsilon1)(q^2-1)^2,$$
    respectively.
    Thus, $(n_1,n_2)\mid (3,q^2-1)fq^{10}(q^5-\epsilon1)(q^3-\epsilon1)(q^2-1)^2$. Hence, $k+1$ divides $3f(q^5-\epsilon1)(q^3-\epsilon1)(q^2-1)^2$ for $(k+1,q)=1$. It follows that $$v=k^2<3^2f^2(q^5-\epsilon1)^2(q^3-\epsilon1)^2(q^2-1)^4<2^43^2f^2q^{24},$$ which implies a contradiction.

    Case (2).\ If $\Soc(G_x)=F_4(q)$, then $|G_x|=c|\Soc(G_x)|$ where $c\mid (2,p)e$.
    Thus,
    $$|G_x \cap X|=\frac{c}{f}q^{24}(q^{12}-1)(q^8-1)(q^6-1)(q^2-1),$$
    and $v=\frac{f}{cd}q^{12}(q^9-\epsilon1)(q^{5}-\epsilon1).$
    According to \cite[Theorems 1 and 10]{MR1204065}, there are two subdegrees dividing
    $$n_1=q^8(q^8+q^4+1),\ n_2=q^{12}(q^8-1)(q^4-1).$$
    Thus, $(n_1,n_2)\mid 3q^8$. Hence, $k+1\mid 3$ for $(k+1,q)=1$. It implies a contradiction for $v=k^2<3^2$.

    Case (3).\ Suppose $G_x \cap X=(4,q-\epsilon1)\cdot (D_5^{\epsilon}(q)\times (\frac{q-\epsilon1}{(4,q-\epsilon1)}))\cdot (4,q-\epsilon1)$. Then $$|G_x\cap X|=q^{20}(q^8-1)(q^6-1)(q^5-\epsilon1)(q^4-1)(q^2-1)(q-\epsilon1),$$ and $$v=q^{16}(q^{12}-1)(q^9-\epsilon1)/d(q^4-1)(q-\epsilon1).$$
    If $(\epsilon,q)=(-,2)$, then $v=2^{16}\cdot 3^2\cdot 7\cdot 13\cdot 19$, a contradiction.
    If $(\epsilon,q)\neq(-,2)$, then
    according to Lemma \ref{le7}, there are two subdegrees dividing $$n_1=(4,q^4-1)fq^{8}(q^{5}-\epsilon1)(q^4+1)(q-\epsilon1), \ n_2=(5,q-\epsilon1)fq^{10}(q^{8}-1)(q^3+\epsilon1)(q-\epsilon1),$$
    respectively.
    Thus, $(n_1,n_2)\mid (5,q-\epsilon1)fq^{8}(q^{8}-1)(q^3+\epsilon1)(q-\epsilon1)$.
    Hence, $k+1\mid 5f(q^{8}-1)(q^3+\epsilon1)(q-\epsilon1)$. Then we get $$q^{16}(q^{12}-1)(q^9-\epsilon1)<5^2f^2(q^{8}-1)^2(q^4-1)(q^3+\epsilon1)^2(q-\epsilon1)^3$$ for $v=k^2$, which is impossible.

    Case (4).\ If $\Soc(G_x)$ = $C_4(q)$ with $q$ odd. By \cite[Remark 4.2]{MR4401918}, $G_x\cap X\cong PSp_8(q)\cdot 2$ in this case. Then
    $$|G_x\cap X|=q^{16}(q^8-1)(q^6-1)(q^4-1)(q^2-1),$$ and
    $$v=q^{20}(q^{12}-1)(q^9-\epsilon1)(q^5-\epsilon1)/d(q^4-1).$$
    According to Lemma \ref{le7}, there exists a subdegree dividing $$n=(2,q-1)^2fq^{8}(q^4+q^2+1)(q^4+1).$$
    Thus, $k+1\mid 2^2f(q^4+q^2+1)(q^4+1)$ for $(k+1,q)=1$. Then we get $$q^{45}<q^{20}(q^{12}-1)(q^9-\epsilon1)(q^5-\epsilon1)<2^4f^2(q^4+q^2+1)^2(q^4+1)^2\cdot d(q^4-1)<2^83f^2q^{20}$$ for $v=k^2$, which is a contradiction.

    Case (5).\ If $G_x \cap X=(2,q-1)^2\cdot (D_4(q)\times (\frac{q-\epsilon1}{(2,q-1)})^2)\cdot (2,q-1)^2\cdot S_3$ with $(\epsilon,q)\neq (+,2)$, then $$|G_x \cap X|=6q^{12}(q^{6}-1)(q^4-1)^2(q^2-1)(q-\epsilon1)^2,$$ and $$v=q^{24}(q^{12}-1)(q^9-\epsilon1)(q^8-1)(q^5-\epsilon1)/6d(q^4-1)^2(q-\epsilon1)^2.$$
    From Lemma \ref{le4}, it follows that $$v<|G_x|^2_{p'}=6^2f^2(q^{6}-1)^2(q^4-1)^4(q^2-1)^2(q-\epsilon1)^4,$$ which forces $(\epsilon,q)=(-,2)$. Then we get $v=2^{23}\cdot 3\cdot 7\cdot 11\cdot 13\cdot 17\cdot 19$, a contradiction.

    Case (6).\ If $G_x \cap X$ = $({^3D_4(q)}\times (q^2+\epsilon q+1))\cdot 3$ with $(\epsilon,q)\neq (-,2)$, then $$|G_x \cap X|=3q^{12}(q^8+q^4+1)(q^{6}-1)(q^2-1)(q^2+\epsilon q+1),$$ and $$v=q^{24}(q^9-\epsilon1)(q^8-1)(q^5-\epsilon1)(q^4-1)/3d(q^2+\epsilon q+1).$$    
    From Lemma \ref{le4}, it follows that $$v<|G_x|^2_{p'}=3^2f^2(q^8+q^4+1)^2(q^{6}-1)^2(q^2-1)^2(q^2+\epsilon q+1)^2,$$ which implies a contradiction.
\end{proof}
\begin{lemma}\label{E_{7}(q)}
	The group $X$ is not $E_{7}(q)$, with $q=p^{e}$.
\end{lemma}
\begin{proof}
    Suppose that $X$ = $E_{7}(q)$ with $q=p^{e}$, where $p$ is a prime. Then $$|G|=f|X|=\frac{f}{d}q^{63}(q^{18}-1)(q^{14}-1)(q^{12}-1)(q^{10}-1)(q^8-1)(q^6-1)(q^2-1),$$ where $f\mid de$, and $d=(2,q-1)$. It is obvious that $f<q$. According to Lemmas \ref{parabolic}-\ref{subfield}, the remaining candidates for $G_x\cap X$ are: $(3,q-\epsilon1)\cdot (E^{\epsilon}_6(q)\times (q-\epsilon1)/(3,q-\epsilon1))\cdot (3,q-\epsilon1)\cdot 2$ with $\epsilon=\pm$ , $d\cdot (A_1(q)\times D_6(q))\cdot d$, $\frac{(4,q-\epsilon1)}{d}\cdot (A^{\epsilon}_7(q)\cdot \frac{(8,q-\epsilon1)}{d}\cdot (2\times \frac{2d}{(4,q-\epsilon1)}))$ with $\epsilon=\pm$, and $\Soc(G_x)=A_1(q)F_4(q)$ with $q>3$.

    Case (1).\ Suppose If $G_x \cap X$ = $(3,q-\epsilon1)\cdot (E^{\epsilon}_6(q)\times (q-\epsilon1)/(3,q-\epsilon1))\cdot (3,q-\epsilon1)\cdot 2$ with $\epsilon=\pm$. Then $$|G_x \cap X|=2q^{36}(q^{12}-1)(q^{9}-\epsilon1)(q^8-1)(q^6-1)(q^5-\epsilon1)(q^2-1)(q-\epsilon1),$$ and $$v=q^{27}(q^{14}-1)(q^9+\epsilon1)(q^5+\epsilon1)/2d(q-\epsilon1).$$
    If $(\epsilon,q)=(-,2)$, then $v=2^{26}\cdot 3\cdot 31\cdot 43\cdot 73\cdot 127$, which is a contradiction.
    If $(\epsilon,q)\neq (-,2)$, then
    according to Lemma \ref{le7}, there are two subdegrees dividing $$n_1=2fq^{12}(q^9-\epsilon1)(q^5-\epsilon1)(q-\epsilon1), \ n_2=2(4,q^5-\epsilon1)fq^{16}(q^9-\epsilon1)(q^8+q^4+1)(q-\epsilon1),$$
    respectively.
    Thus, $(n_1, n_2)\mid 2f\cdot 2 (4,q^5-\epsilon1)q^{12}(q^9-\epsilon1)(q-\epsilon1)$. Hence, $k+1\mid 2^4f(q^9-\epsilon1)(q-\epsilon1)$ for $(k+1,q)=1$. Then we get $$q^{27}(q^{14}-1)(q^9+\epsilon1)(q^5+\epsilon1)<2^9df^2(q^9-\epsilon1)^2(q-\epsilon1)^3$$ for $v=k^2$, which is impossible.

    Case (2).\ If $G_x \cap X$ = $d\cdot (A_1(q)\times D_6(q))\cdot d$, then $$|G_x \cap X|=\frac{1}{d}q^{31}(q^{10}-1)(q^{8}-1)(q^6-1)^2(q^4-1)(q^2-1)^2,$$ and $$v=q^{32}(q^{18}-1)(q^{14}-1)(q^6+1)/(q^4-1)(q^2-1).$$
    According to Lemma \ref{le7}, there exists a subdegree dividing $$n=\frac{(6,q-\epsilon1)}{(2,q-1)}fq^{16}(q^{8}-1)(q^5+\epsilon1)(q^{3}+\epsilon1)(q^2-1),$$
    where $\epsilon=\pm$.
    Thus, $k+1\mid 6f(q^{8}-1)(q^5+\epsilon1)(q^{3}+\epsilon1)(q^2-1)$. Then we get $$q^{70}<2^23^2q^{28}(q^5+\epsilon1)^2(q^{3}+\epsilon1)^2$$ for $v=k^2$, which is a contradiction.

    Case (3).\  Suppose $G_x \cap X$ = $\frac{(4,q-\epsilon1)}{d}\cdot (A^{\epsilon}_7(q)\cdot \frac{(8,q-\epsilon1)}{d}\cdot (2\times \frac{2d}{(4,q-\epsilon1)}))$ with $\epsilon=\pm$, then $$|G_x \cap X|=\frac{4}{d}q^{28}(q^{8}-1)(q^7-\epsilon1)(q^6-1)(q^5-\epsilon1)(q^4-1)(q^3-\epsilon1)(q^2-1),$$ and $$v=q^{35}(q^{18}-1)(q^{12}-1)(q^7+\epsilon1)(q^5+\epsilon1)/4 (q^4-1)(q^3-\epsilon1).$$
        
    If $q=2$ and $\epsilon=\pm$, then $v=2^{33}\cdot 3^6\cdot 7\cdot 11\cdot 13\cdot 19\cdot 43\cdot 73$ or $2^{33}\cdot 3^2\cdot 7^2\cdot 13\cdot 19\cdot 31\cdot 73\cdot 127$ respectively, a contradiction.
    
    If $q>2$ even, then according to Lemma \ref{le7}, there exists a subdegree dividing $$n=4fq^{12}(q^7-\epsilon1)(q^5-\epsilon1)(q^{3}-\epsilon1).$$
    Thus, $k+1\mid 2^2f(q^7-\epsilon1)(q^5-\epsilon1)(q^{3}-\epsilon1)$. Therefore, $$v=k^2<2^4f^2(q^7-\epsilon1)^2(q^5-\epsilon1)^2(q^{3}-\epsilon1)^2,$$ which is a contradiction.
    
    If $q$ odd, then according to Lemma \ref{le7}, there exists a subdegree dividing $$n=8fq^{16}(q^7-\epsilon1)(q^5-\epsilon1)(q^4+1)(q^{3}-\epsilon1).$$
    Thus, $k+1\mid2^3f(q^7-\epsilon1)(q^5-\epsilon1)(q^4+1)(q^{3}-\epsilon1)$. Therefore, $$v=k^2<2^6f^2(q^7-\epsilon1)^2(q^5-\epsilon1)^2(q^4+1)^2(q^{3}-\epsilon1)^2,$$ which is a contradiction.

    Case (4).\ If $\Soc(G_x)$ = $A_1(q)\times F_4(q)$ with $q>3$, then $$|G_x \cap X|=\frac{c}{f}q^{25}(q^{12}-1)(q^8-1)(q^6-1)(q^2-1)^2,$$ and $$v=\frac{f}{cd}q^{36}(q^{12}-1)(q^4+1)(q^2+1),$$ where $c\mid (2,p)de^2$.
    Clearly, 
    \begin{align*}
    &(cd(v-1),q^{12}-1)\mid cd, &(cd(v-1),q^{8}-1)\mid cd,\\ &(cd(v-1),q^{6}-1)\mid cd, &(cd(v-1),q^{2}-1)\mid cd.
    \end{align*}
     Thus, $(|G_x|,v-1)\mid c\cdot {(cd)}^5$. Hence, $k+1\mid c^6d^5$. Therefore, $v=k^2<c^{12}d^{10}$, which is a contradiction.
\end{proof}
\begin{lemma}\label{E_{8}(q)}
	The group $X$ is not $E_{8}(q)$, with $q=p^{e}$.
\end{lemma}
\begin{proof}
    Suppose that $X$ = $E_{8}(q)$ with $q=p^{e}$, where $p$ is a prime. Then $$|G|=f|X|=fq^{120}(q^{30}-1)(q^{24}-1)(q^{20}-1)(q^{18}-1)(q^{14}-1)(q^{12}-1)(q^8-1)(q^2-1),$$ where $f\mid e$. It is obvious that $f<q$.  According to Lemmas  \ref{parabolic}-\ref{subfield}, the remaining candidates for $G_x\cap X$ are: $(2,q-1)\cdot (A_1(q)\times E_7(q))\cdot (2,q-1)$, $(2,q-1)\cdot D_8(q)\cdot (2,q-1)$, and $(3,q-\epsilon1)\cdot (A_2^{\epsilon}(q)\times E_6^{\epsilon}(q))\cdot (3,q-\epsilon1)\cdot 2$ with $\epsilon=\pm$.

    Case (1).\ If $G_x \cap X$ = $(2,q-1)\cdot (A_1(q)\times E_7(q))\cdot (2,q-1)$, then $$|G_x \cap X|=q^{64}(q^{18}-1)(q^{14}-1)(q^{12}-1)(q^{10}-1)(q^8-1)(q^6-1)(q^2-1)^2,$$ and $$v=q^{56}(q^{30}-1)(q^{24}-1)(q^{20}-1)/(q^{10}-1)(q^6-1)(q^2-1).$$
    According to Lemma \ref{le7}, there are two subdegrees dividing $$n_1=(2,q-1)^4fq^{32}(q^{18}-1)(q^{14}-1)(q^6+1)/(q^4-1)(q^2-1),$$
    and $$n_2=(3,q-\epsilon1)fq^{28}(q^{14}-1)(q^9+\epsilon1)(q^5+\epsilon1)(q^2-1),$$
    respectively.
    Thus, $(n_1,n_2)\mid (3,q-\epsilon1)fq^{28}(q^{14}-1)(q^9+\epsilon1)(q^5+\epsilon1)(q^2-1)$. Therefore, $k+1$ divides $3f(q^{14}-1)(q^9+\epsilon1)(q^5+\epsilon1)(q^2-1)$ for $(k+1,p)=1$. Then we get $$q^{112}<3^2f^2(q^{14}-1)^2(q^9+\epsilon1)^2(q^5+\epsilon1)^2(q^2-1)^2$$ for $v=k^2$, which is a contradiction.

    Case (2).\ If $G_x \cap X$ = $(2,q-1)\cdot D_8(q)\cdot (2,q-1)$, then $$|G_x \cap X|=q^{56}(q^{14}-1)(q^{12}-1)(q^{10}-1)(q^8-1)^2(q^6-1)(q^4-1)(q^2-1),$$ and $$v=q^{64}(q^{30}-1)(q^{24}-1)(q^{20}-1)(q^{18}-1)/(q^{10}-1)(q^8-1)(q^6-1)(q^4-1).$$
    Clearly, $$(q(q^8+q^6+q^4+q^2+1)(q^4+q^2+1)(q^4-q^2+1),v-1)=1.$$ 
    Then, $$(|G_x|,v-1)=(f(q^{14}-1)(q^4+1)^2(q^2+1)^4(q^2-1)^7,v-1).$$
    Thus, $k+1\mid f(q^{14}-1)(q^4+1)^2(q^2+1)^4(q^2-1)^7$. Therefore, $$v=k^2<f^2(q^{14}-1)^2(q^4+1)^4(q^2+1)^8(q^2-1)^{14},$$ which implies a contradiction.

    Case (3).\ If $G_x \cap X$ = $(3,q-\epsilon1)\cdot (A_2^{\epsilon}(q)\times E_6^{\epsilon}(q))\cdot (3,q-\epsilon1)\cdot 2$ with $\epsilon=\pm$, then $$|G_x \cap X|=2q^{39}(q^{12}-1)(q^{9}-\epsilon1)(q^{8}-1)(q^6-1)(q^5-\epsilon1)(q^3-\epsilon1)(q^2-1)^2,$$ and $$v=q^{81}(q^{30}-1)(q^{24}-1)(q^{20}-1)(q^{18}-1)(q^{14}-1)/2(q^{9}-\epsilon1)(q^6-1)(q^5-\epsilon1)(q^3-\epsilon1)(q^2-1).$$
    From Lemma \ref{le4}, it follows that $$v<|G_x|^2_{p'}=2^2f^2(q^{12}-1)^2(q^{9}-\epsilon1)^2(q^{8}-1)^2(q^6-1)^2(q^5-\epsilon1)^2(q^3-\epsilon1)^2(q^2-1)^4.$$ Then we get $$q^{186}<2^{2}f^2q^{68}(q^{9}-\epsilon1)^3(q^5-\epsilon1)^3(q^3-\epsilon1)^3<2^{11}q^{121},$$ which is a contradiction.
\end{proof}
\medskip
    \noindent{\bf Proof\ of\ Theorem\ \ref{theorem1}} The result is a direct consequence of Lemmas \ref{^{2}B_{2}(q)}-\ref{E_{8}(q)}. ~~~~~~~~~~~~~~$ \square$
\section*{Declaration of competing interest}
The authors declare that they have no known competing financial interests or personal relationships that could have appeared to influence the work reported in this paper.
\section*{Data availability} 
No date was used for the research described in the article.

\end{document}